\documentclass[11pt]{article}
\usepackage[T1]{fontenc}
\usepackage{lmodern}
\usepackage{amsmath,amssymb,amsthm,mathtools}
\usepackage[margin=1.05in]{geometry}
\usepackage{microtype}
\usepackage{enumitem}
\usepackage{needspace}
\usepackage{xcolor}
\usepackage[colorlinks=true,linkcolor=blue!55!black,citecolor=blue!55!black,
  urlcolor=blue!55!black,breaklinks=true]{hyperref}
\hypersetup{pdftitle={Order-preserving maps of the Turing degrees: regularity and consistency strength},
  pdfsubject={Order-preserving Martin conjecture, Solovay models, and equiconsistency with an inaccessible cardinal}}
\ifdefined\pdfgentounicode\input{glyphtounicode}\pdfgentounicode=1\fi
\setlist[enumerate]{leftmargin=*,itemsep=4pt,topsep=5pt}
\setlist[itemize]{leftmargin=*,itemsep=4pt,topsep=5pt}
\numberwithin{equation}{section}
\newtheorem{theorem}{Theorem}[section]
\newtheorem{proposition}[theorem]{Proposition}
\newtheorem{lemma}[theorem]{Lemma}
\newtheorem{corollary}[theorem]{Corollary}
\theoremstyle{definition}
\newtheorem{definition}[theorem]{Definition}
\theoremstyle{remark}
\newtheorem{remark}[theorem]{Remark}
\newcommand{\R}{\mathbb R}

\newcommand{\Cantor}{2^\omega}
\newcommand{\DT}{\mathcal D_T}
\newcommand{\leT}{\le_T}
\newcommand{\geT}{\ge_T}
\newcommand{\equivT}{\equiv_T}
\newcommand{\degT}[1]{[#1]_T}
\newcommand{\cone}[1]{\mathcal C_{#1}}
\newcommand{\ACR}{\mathsf{AC}_{\omega}(\R)}
\newcommand{\DCR}{\mathsf{DC}_{\R}}
\newcommand{\PSP}{\mathsf{PSP}_{\R}}
\newcommand{\CU}{\mathsf{CU}}
\newcommand{\OPreal}{\mathsf{OP}_{\mathrm{real}}}
\newcommand{\OPdeg}{\mathsf{OP}_{\mathrm{deg}}}
\DeclareMathOperator{\Con}{Con}
\DeclareMathOperator{\ran}{ran}
\DeclareMathOperator{\id}{id}
\title{Order-preserving maps of the Turing degrees:\\
regularity and consistency strength}
 \author{Liang Yu\thanks{Yu was partially supported by NSF of China
No. 12025103 and the Natural Science Foundation of Jiangsu Province
No. BK20243060.}\\[1ex]
\small School of Mathematics\\
\small Nanjing University, Jiangsu Province 210093\\
\small P. R. of China\\
\small\href{mailto:yuliang.nju@gmail.com}{\texttt{yuliang.nju@gmail.com}}}

\begin{document}
\maketitle

\begin{abstract}
We formulate a proof of Part~I of Martin's conjecture for
order-preserving maps directly on the quotient of the reals by Turing
equivalence. The assumptions are countable choice for reals, the perfect-set
property for sets of reals, and continuous uniformization on a comeager
domain.  The argument applies
under $\mathsf{AD}+\DCR$ and inside the choiceless Solovay model.
We then determine the exact consistency strength of the unrestricted
order-preserving dichotomy as an axiom over $\mathsf{ZF}+\ACR$:
it is equiconsistent with $\mathsf{ZFC}$ plus a strongly inaccessible
cardinal.  
\end{abstract}



\section{Introduction and main statement}\label{sec:introduction}

Let $\DT$ denote the partially ordered set of Turing degrees. We write
$\mathbf a\vee\mathbf b$ for the join of two degrees and
$\degT{X}$ for the degree of a real $X\in\Cantor$. For
$\mathbf a\in\DT$, put
\[
 \cone{\mathbf a}=\{\mathbf d\in\DT:\mathbf a\leT\mathbf d\}.
\]
A map $\varphi:\DT\to\DT$ is \emph{order preserving} if
$\mathbf d\leT\mathbf e$ implies
$\varphi(\mathbf d)\leT\varphi(\mathbf e)$. It is \emph{constant
on a cone} if its values are equal throughout some
$\cone{\mathbf a}$, and it is \emph{above the identity on a cone}
if $\mathbf d\leT\varphi(\mathbf d)$ throughout some cone.
The latter relation is also written $\id\le_M\varphi$.
A set of degrees is \emph{cofinal} if it meets every cone.

Part~I of Martin's conjecture asserts, in its usual determinacy setting,
that a Turing-invariant real function is constant in degree on a cone or
above the identity on a cone. Lutz--Siskind establish this conclusion for
order-preserving functions \cite{LutzSiskind2025}. We give a formulation
under explicit regularity hypotheses, with the function defined on
$\DT$ itself. In a choiceless setting this avoids assuming, without
justification, that every degree map has a global real lift.

\Needspace{12\baselineskip}
\begin{definition}\label{def:principles}
We use the following principles.
\begin{enumerate}[label=(\roman*)]
 \item $\ACR$ is countable choice for nonempty sets of reals.
 \item $\PSP$ asserts that every uncountable subset of $\Cantor$
 contains a nonempty perfect subset.
 \item $\CU$ asserts that for every relation
 $S\subseteq\Cantor\times\Cantor$ with
 $\forall X\,\exists Y\ S(X,Y)$, there are a comeager
 $G_\delta$ set $C\subseteq\Cantor$ and a continuous
 $g:C\to\Cantor$ such that $S(X,g(X))$ for every $X\in C$.
\end{enumerate}
Continuity and the $G_\delta$ domain are part of the definition of
$\CU$ in this paper.
\end{definition}

\begin{theorem}\label{thm:main}
Work in $\mathsf{ZF}+\ACR+\PSP+\CU$. Every order-preserving
map $\varphi:\DT\to\DT$ is constant on a cone or above the
identity on a cone. If $\varphi$ is not constant on any cone:
\begin{enumerate}[label=(\roman*)]
 \item there is $\mathbf k\in\DT$ such that
 \begin{equation}\label{eq:main-code}
  \mathbf d\leT\mathbf k\vee\varphi(\mathbf d)
  \qquad(\mathbf d\geT\mathbf k);
 \end{equation}
 \item the range of $\varphi$ is cofinal;
 \item for every $\mathbf a\geT\mathbf k$ and every degree
 $\mathbf v$, the set
 \begin{equation}\label{eq:main-fiber}
  \{X\in\Cantor:\varphi(\mathbf a\vee\degT{X})=\mathbf v\}
 \end{equation}
 is at most countable;
 \item there is $\mathbf b\geT\mathbf k$ such that
 \begin{equation}\label{eq:main-domination}
  \mathbf d\leT\varphi(\mathbf d)
  \qquad(\mathbf d\geT\mathbf b).
 \end{equation}
 Thus $\varphi$ is above the identity on the upper cone
 $\cone{\mathbf b}$.
\end{enumerate}
\end{theorem}

\begin{remark}[The Lutz--Siskind theorem as a consequence]
\label{rem:LS-consequence}
Under $\mathsf{ZF}+\mathsf{AD}+\DCR$, the regularity hypotheses
of Theorem~\ref{thm:main} hold; see Theorem~\ref{cor:AD}.
If $f:\Cantor\to\Cantor$ is Turing invariant and order preserving,
the induced degree map
\[
 \varphi(\degT{X})=\degT{f(X)}
\]
is well defined and order preserving. Theorem~\ref{thm:main}
therefore shows that $f$ is constant in Turing degree on a cone or
satisfies $X\leT f(X)$ throughout an upper cone.
Thus Theorem~\ref{thm:main} implies the order-preserving case of
Part~I of Martin's conjecture proved by Lutz--Siskind
\cite[Theorem~4.8]{LutzSiskind2025}.
The cofinality argument in Section~\ref{sec:cofinality} uses
Groszek--Slaman's original perfect-set lemma and relative highness;
it does not invoke the strengthened coding theorem
\cite[Theorem~4.3]{LutzSiskind2025}.
\end{remark}

\begin{theorem}[Consistency strength]\label{thm:consistency}
Let $B$ be any one of $\ACR$, $\DCR$, and $\mathsf{DC}$.
The following theories are equiconsistent:
\begin{enumerate}[label=(\roman*)]
 \item $\mathsf{ZFC}$ plus the existence of a strongly inaccessible cardinal;
 \item $\mathsf{ZF}+B$ plus the assertion that every Turing-invariant,
 order-preserving function $f:\Cantor\to\Cantor$ is constant in
 Turing degree on a cone or satisfies $X\leT f(X)$ on a cone;
 \item $\mathsf{ZF}+B$ plus the assertion that every order-preserving
 map $\varphi:\DT\to\DT$ is constant on a cone or above the
 identity on a cone.
\end{enumerate}
Moreover, the theory in~\textup{(ii)} proves that
$\omega_1^V$ is a strongly inaccessible cardinal in $L$.
\end{theorem}

Section~\ref{sec:consistency} proves Theorem~\ref{thm:consistency}.
The upper bound uses the regularity argument in the Solovay model.
For the lower bound, we adapt the counterexample of
Lutz--Siskind \cite[Theorem~6.11]{LutzSiskind2025} inside $L[r]$ and
extend its defining rule to all ambient reals. This forces
$\omega_1^{L[r]}<\omega_1$ for every real $r$ whenever the unrestricted
dichotomy holds. Countable choice for reals then gives the
inaccessible cardinal in $L$.

The proof of Theorem~\ref{thm:main} has two distinct ingredients.
Continuous uniformization over a
perfect family of outputs yields \eqref{eq:main-code}; a single code
accounts for both the input construction and output decoding.
Two applications of Groszek--Slaman's perfect-set lemma
\cite[Lemma~2.2]{GroszekSlaman1998}, connected by Martin's high
domination theorem \cite{Martin1966}, prove that every countably
directed set containing a perfect subset is cofinal. This makes the
fixed code computable from one value of $\varphi$, and monotonicity
completes the argument. Section~\ref{sec:cofinality} includes the
relative highness argument and the resulting cofinality proof.

This organization is closely related to the continuous-decoder argument
associated with Kihara in \cite[Section~4.3]{LutzSiskind2025}.
Sections~2--5 make the quotient formulation and the regularity
assumptions explicit; Section~6 derives the matching consistency bounds.
The order-preserving dichotomy and the cofinality conclusion are due to
Lutz--Siskind. The proof of cofinality given here uses the original
Groszek--Slaman lemma in place of their strengthened coding theorem.

Section~\ref{sec:main-proof} proves the main theorem. The applications in
Section~\ref{sec:applications} include $\mathsf{AD}+\DCR$ and the
choiceless Solovay model.

\section{Coding continuous maps and controlling the range}
\label{sec:preliminaries}

Unless a stronger hypothesis is specified, this section works in
$\mathsf{ZF}+\ACR$. All real oracles in the main argument belong
to $\Cantor$. Write $[s]=\{X:s\subseteq X\}$ for a basic clopen
cylinder, and fix a computable enumeration $(\Phi_e)_{e\in\omega}$
of oracle Turing machines.

The following lemmas are clear.

\begin{lemma}[A single code for a continuous map]\label{lem:code}
If $D\subseteq\Cantor$ and $g:D\to\Cantor$ is continuous in the
subspace topology, then some real $K$ and one oracle program $\Phi_e$ satisfy
\[
 g(X)=\Phi_e^{ K\oplus X}\qquad(X\in D).
\]
 
\end{lemma}

\begin{lemma}[Cantor subsets]\label{lem:cantor-subset}
Every nonempty perfect subset of $\Cantor$ contains the image of a
continuous injection from $\Cantor$. Every comeager $G_\delta$
subset of $\Cantor$ also contains such an image.
\end{lemma}
The following lemma can be found in {\cite[Theorem 8.38]{Kechris1995}}.
 
\begin{lemma}[Continuous restriction]\label{lem:continuous-restriction}
Every Baire-measurable function $g:X\to\Cantor$, where $X$ is
Polish, is continuous on a comeager $G_\delta$ subset of $X$.
\end{lemma}

\begin{definition}\label{def:saturated-range}
For a degree map $\varphi$, its \emph{saturated real range} is
\[
 R_\varphi=\{Y\in\Cantor:\degT{Y}\in\ran(\varphi)\}.
\]
This set is defined without choosing representatives of all
degrees.
\end{definition}

\Needspace{9\baselineskip}
\begin{lemma}[Countable range and cone constancy]\label{lem:range}
For an order-preserving map $\varphi:\DT\to\DT$, the following
conditions are equivalent:
\begin{enumerate}[label=(\roman*)]
 \item $\varphi$ is constant on a cone;
 \item its range is bounded by one degree;
 \item $R_\varphi$ is countable.
\end{enumerate}
\end{lemma}

\begin{proof}
 Clearly (i) implies (ii) and (ii) implies (iii).  
 
Assume (iii), and enumerate the nonempty set $R_\varphi$ as
$(Y_n)_{n\in\omega}$, with repetitions if necessary. Choose
$X_n\in\Cantor$ such that
$\varphi(\degT{X_n})=\degT{Y_n}$, using $\ACR$.
Put $B=\bigoplus_nX_n$ and $\mathbf b=\degT{B}$.
Every attained degree is below $\varphi(\mathbf b)$, by
monotonicity. For $\mathbf d\geT\mathbf b$,
\[
 \varphi(\mathbf b)\leT\varphi(\mathbf d)
 \leT\varphi(\mathbf b),
\]
which gives (i).
\end{proof}

\Needspace{8\baselineskip}
\begin{lemma}[Countable directedness]\label{lem:directed}
If $\varphi$ is order preserving, every countable subset of
$R_\varphi$ has an upper bound for Turing reducibility that belongs
to $R_\varphi$.
\end{lemma}

\begin{proof}
Given $Y_n\in R_\varphi$, choose $X_n$ with
$\varphi(\degT{X_n})=\degT{Y_n}$. Let
$\mathbf b=\degT{\bigoplus_nX_n}$ and choose one representative
$B$ of $\varphi(\mathbf b)$. Then $B\in R_\varphi$ and
$Y_n\leT B$ for every $n$.
\end{proof}

\begin{remark}\label{rem:nonuniform}
The conclusion of Lemma~\ref{lem:directed} is a pointwise upper bound.
It does not assert that $B$ computes the sequence $(Y_n)$ uniformly,
or that $\bigoplus_nY_n\leT B$. Section~\ref{sec:cofinality}
addresses this distinction by passing to an upper bound that is high
relative to $B$ and computing a suitable covering sequence from it.
\end{remark}

\section{Cofinality from the Groszek--Slaman lemma}
\label{sec:cofinality}

This section works in $\mathsf{ZF}+\ACR$. We prove the
cofinality conclusion needed below using Groszek--Slaman's original
perfect-set lemma and an elementary consequence of relative highness.
No regularity assumption is needed within this section.

For a perfect set $P\subseteq\Cantor$, a point $Y\in P$ is
\emph{eventually constant in $P$} if, for some initial segment $s$ of
$Y$, it is the lexicographically least or greatest member of $P\cap[s]$.
These are precisely the branches that eventually always take the
leftmost or always take the rightmost available continuation in the
perfect tree for $P$.

\begin{theorem}[Groszek--Slaman, {\cite[Lemma~2.2]{GroszekSlaman1998}}]
\label{thm:GS}
Let $P\subseteq\Cantor$ be nonempty and perfect, and let
$B\subseteq P$ be countable and dense, with no member eventually
constant in $P$. Suppose a sequence $(G_i)_{i\in\omega}$ of reals
has range containing $B$. For every $X\in\Cantor$, there are
$Y_0,Y_1\in P$ such that
\begin{equation}\label{eq:GS}
 X\leT\left(\bigoplus_{i\in\omega}G_i\right)\oplus Y_0\oplus Y_1.
\end{equation}
\end{theorem}

The sequence in Theorem~\ref{thm:GS} may contain additional reals
outside $P$. This feature is essential below: we construct uniform
sequences containing every real computable from a given oracle, rather
than compute an enumeration of the prescribed dense set itself.

The following lemma is due to Jockusch. We include a proof for completeness.

\begin{lemma}[Uniform covering from relative highness, Jockusch {\cite[Theorem~1]{Jockusch1972}}]
\label{lem:high-cover}
Suppose $C,D\in\Cantor$ satisfy $C\leT D$ and $C''\leT D'$.
There is a sequence $(G_i)_{i\in\omega}$ with
$\bigoplus_iG_i\leT D$ whose range contains every $C$-computable
real in $\Cantor$.
\end{lemma}

\begin{proof}
We first give the needed direction of Martin's high domination theorem
\cite{Martin1966}, relativized to $C$. Let
\[
 \mathrm{TOT}^{C}
 =\{e:\Phi_e^C\text{ is a total function from }\omega\text{ to }\omega\}.
\]
Since $\mathrm{TOT}^{C}\leT C''\leT D'$, the relativized limit
lemma gives a total $D$-computable function
$p:\omega^2\to\{0,1\}$ such that
\[
 \lim_s p(e,s)=\chi_{\mathrm{TOT}^{C}}(e).
\]
For each $n$, let $s_n$ be the least $s\ge n$ such that
\[
 \forall e\le n\,
 \bigl(p(e,s)=1\ \Longrightarrow\ \Phi_e^C(n)[s]\mathord\downarrow\bigr).
\]
This search is computable from $D$, since $C\leT D$. It terminates:
the finitely many approximations with $e\le n$ eventually stabilize,
after which the indices still marked $1$ all give total computations.
Define
\[
 h(n)=1+\max\bigl(\{n\}\cup
       \{\Phi_e^C(n):e\le n\text{ and }p(e,s_n)=1\}\bigr).
\]
Then $h\leT D$. If $\Phi_e^C$ is total, choose $N\ge e$ such
that $p(e,s)=1$ for every $s\ge N$. For $n\ge N$, we have
$s_n\ge n\ge N$, so $h(n)>\Phi_e^C(n)$. Thus $h$ eventually
dominates every total $C$-computable function. This argument places
no computable-enumerability assumption on $C$ or $D$.

For each $e\in\omega$ and $\sigma\in2^{<\omega}$, define
\[
 G_{e,\sigma}(n)=
 \begin{cases}
  \sigma(n),&n<|\sigma|,\\
  \Phi_e^C(n),&n\ge|\sigma|\text{ and }\Phi_e^C(n)
       \text{ halts within }h(n)\text{ steps with a binary output},\\
  0,&\text{otherwise}.
 \end{cases}
\]
These are total binary reals uniformly computable from $D$.
If $Z=\Phi_e^C\in\Cantor$, the running-time function of this
computation is total and $C$-computable. Hence there is $N$ such
that the computation on input $n$ halts within $h(n)$ steps whenever
$n\ge N$. Taking $\sigma=Z\upharpoonright N$ gives
$G_{e,\sigma}=Z$. Enumerate the pairs $(e,\sigma)$ in a fixed
computable order to obtain the required sequence.
\end{proof}

The following conclusion is Lutz--Siskind's cofinality corollary.
The proof here derives it from Theorem~\ref{thm:GS} and
Lemma~\ref{lem:high-cover}.

\begin{corollary}[{\cite[Corollary~4.5]{LutzSiskind2025}}]
\label{cor:perfect-directed}
If $A\subseteq\Cantor$ contains a nonempty perfect set and every
countable subset of $A$ has a Turing upper bound in $A$, then $A$
is cofinal in the Turing degrees.
\end{corollary}

\begin{proof}
Choose a perfect $P\subseteq A$. The points eventually constant in
$P$ form a countable set: each is the least or greatest point of
$P\cap[s]$ for some $s\in2^{<\omega}$. Every nonempty
$P\cap[s]$ is perfect and uncountable. Using $\ACR$, choose a
countable dense set $B\subseteq P$ avoiding all these points.
Let $C\in A$ compute every member of $B$.

\emph{First application of Groszek--Slaman.}
For each oracle-program index $e$, put
\[
 H_e(n)=
 \begin{cases}
  \Phi_e^C(n),&\text{if this computation halts with a binary output},\\
  0,&\text{otherwise}.
 \end{cases}
\]
The sequence $(H_e)$ is uniformly $C'$-computable and contains every
$C$-computable binary real, hence contains $B$. Apply
Theorem~\ref{thm:GS} with this sequence and target $C''$ to obtain
$U_0,U_1\in P$ such that
\begin{equation}\label{eq:GS-highness}
 C''\leT\left(\bigoplus_eH_e\right)\oplus U_0\oplus U_1
       \leT C'\oplus U_0\oplus U_1.
\end{equation}
Choose an upper bound $D\in A$ for $C,U_0,U_1$. Then
$C\leT D$ and $C''\leT D'$. Lemma~\ref{lem:high-cover}
therefore supplies a uniformly $D$-computable sequence $(G_i)$
containing every $C$-computable real, and in particular containing $B$.

\emph{Second application of Groszek--Slaman.}
Given any target $X\in\Cantor$, apply Theorem~\ref{thm:GS}
again, now with $(G_i)$. There are $V_0,V_1\in P$ such that
\[
 X\leT\left(\bigoplus_iG_i\right)\oplus V_0\oplus V_1
       \leT D\oplus V_0\oplus V_1.
\]
An upper bound in $A$ for $D,V_0,V_1$ computes $X$. Thus $A$
is cofinal.
\end{proof}

\begin{remark}[The role of highness and of the dense set]
The first application obtains $C''\leT D'$, not $C'\leT D$.
Relative highness nevertheless makes a covering sequence computable
from $D$, which is exactly what the second application needs. The
set $A$ need not be closed under Turing jump or under actual joins;
upper bounds in $A$ suffice. Also, we choose the dense set $B$
before choosing its bound $C$, so the restriction on eventually
constant branches is satisfied. The argument does not require the
stronger coding statement for an arbitrary prescribed dense set.
\end{remark}

\section{A continuous decoder and the proof of the main theorem}
\label{sec:main-proof}

This section is closely related to the continuous-decoder argument
associated with Kihara in \cite[Section~4.3]{LutzSiskind2025}.

\begin{proposition}[Continuous coding of a perfect family of outputs]
\label{prop:factorization}
Assume $\PSP+\CU$, and let $\varphi:\DT\to\DT$ have
uncountable saturated real range. There are continuous functions
$\alpha,\beta:\Cantor\to\Cantor$, with $\beta$ injective,
such that
\begin{equation}\label{eq:factorization}
 \varphi(\degT{\alpha(X)})=\degT{\beta(X)}
 \qquad(X\in\Cantor).
\end{equation}
No order-preservation assumption is needed in this proposition.
\end{proposition}

\begin{proof}
Choose a perfect $P\subseteq R_\varphi$ and a continuous
injection $h:\Cantor\to P$, using
Lemma~\ref{lem:cantor-subset}. The relation
\[
 S(X,U)\quad\Longleftrightarrow\quad
 \varphi(\degT{U})=\degT{h(X)}
\]
has a nonempty section at every $X$, because $h(X)\in R_\varphi$.
By $\CU$, there are a comeager $G_\delta$ set $C$ and a
continuous $g:C\to\Cantor$ with
\[
 \varphi(\degT{g(X)})=\degT{h(X)}\qquad(X\in C).
\]
Choose a continuous injection $q:\Cantor\to C$ and set
$\alpha=g\circ q$ and $\beta=h\circ q$. These maps have the
required properties.
\end{proof}

\begin{proposition}[Domination with one fixed code]\label{prop:fixed-code}
Assume $\PSP+\CU$. If $\varphi$ is order preserving and
$R_\varphi$ is uncountable, there is $\mathbf k$ satisfying
\eqref{eq:main-code}.
\end{proposition}

\begin{proof}
Take $\alpha,\beta$ from Proposition~\ref{prop:factorization}.
Since $\Cantor$ is compact and $\beta$ is continuous and
injective, its inverse is continuous on the compact image
$\beta[\Cantor]$. Apply Lemma~\ref{lem:code} to $\alpha$
and to this inverse, and join the two codes into a single real
$K$. Then
\begin{equation}\label{eq:two-codes}
 \alpha(X)\leT K\oplus X,
 \qquad X\leT K\oplus\beta(X)
 \qquad(X\in\Cantor).
\end{equation}
Put $\mathbf k=\degT{K}$.

Fix $\mathbf d\geT\mathbf k$, and take any representative
$X\in\mathbf d$. The first inequality in \eqref{eq:two-codes}
gives $\degT{\alpha(X)}\leT\mathbf d$. Hence, by monotonicity
and \eqref{eq:factorization},
\[
 \degT{\beta(X)}=\varphi(\degT{\alpha(X)})
 \leT\varphi(\mathbf d).
\]
The second inequality in \eqref{eq:two-codes} now yields
\[
 \mathbf d\leT\mathbf k\vee\degT{\beta(X)}
 \leT\mathbf k\vee\varphi(\mathbf d).
\]
 
\end{proof}

\begin{corollary}[Countable fibers after normalization]\label{cor:fibers}
Under the hypotheses of Proposition~\ref{prop:fixed-code}, fix
$\mathbf a\geT\mathbf k$. For every degree $\mathbf v$, the set
in \eqref{eq:main-fiber} is at most countable. In particular, any partial
real representative of $X\mapsto\varphi(\mathbf a\vee\degT{X})$
has countable point fibers.
\end{corollary}

\begin{proof}
Fix $Y\in\mathbf v$. Whenever
$\varphi(\mathbf a\vee\degT{X})=\mathbf v$, apply
\eqref{eq:main-code} to $\mathbf a\vee\degT{X}$ to obtain
\[
 \degT{X}\leT\mathbf a\vee\degT{X}
 \leT\mathbf k\vee\mathbf v.
\]
Thus every such $X$ is computable from $K\oplus Y$, and the
set of those reals is countable.
\end{proof}

\begin{lemma}[Removing the fixed code]\label{lem:absorption}
Suppose $\varphi$ is order preserving,
$\mathbf d\leT\mathbf k\vee\varphi(\mathbf d)$ whenever
$\mathbf d\geT\mathbf k$, and there is a degree $\mathbf e$
with $\mathbf k\leT\varphi(\mathbf e)$. Then
\[
 \mathbf d\leT\varphi(\mathbf d)
 \qquad(\mathbf d\geT\mathbf e\vee\mathbf k).
\]
\end{lemma}

\begin{proof}
For such $\mathbf d$, monotonicity gives
$\mathbf k\leT\varphi(\mathbf e)\leT\varphi(\mathbf d)$.
Therefore
$\mathbf d\leT\mathbf k\vee\varphi(\mathbf d)=\varphi(\mathbf d)$.
\end{proof}

\begin{proof}[Proof of Theorem~\ref{thm:main}]
If $R_\varphi$ is countable, Lemma~\ref{lem:range} gives
constancy on a cone. Suppose it is uncountable. Proposition
\ref{prop:fixed-code} gives $\mathbf k$ and
\eqref{eq:main-code}; Corollary~\ref{cor:fibers} gives the
fiber assertion. By $\PSP$, the set $R_\varphi$ contains a
perfect subset, and Lemma~\ref{lem:directed} makes it countably
directed. Corollary~\ref{cor:perfect-directed} therefore makes
$R_\varphi$, and hence the degree range of $\varphi$, cofinal.
Choose $\mathbf e$ with $\mathbf k\leT\varphi(\mathbf e)$ and put
$\mathbf b=\mathbf e\vee\mathbf k$. Then $\mathbf b\geT\mathbf k$,
and Lemma~\ref{lem:absorption} proves \eqref{eq:main-domination}
on $\cone{\mathbf b}$, establishing (iv) and the asserted dichotomy.
\end{proof}

\begin{remark}[Only the decoder needs to be reached]\label{rem:single-target}
Full cofinality is more than is needed to apply
Lemma~\ref{lem:absorption}. In the proof of
Corollary~\ref{cor:perfect-directed}, take $A=R_\varphi$ and
use $X=K$ in the second application of Groszek--Slaman. The resulting
upper bound $E\in R_\varphi$ computes $K$. By the definition of
$R_\varphi$, some degree $\mathbf e$ satisfies
$\varphi(\mathbf e)=\degT{E}$, and hence
$\mathbf k\leT\varphi(\mathbf e)$. This is precisely the
hypothesis needed for absorption. Only order preservation is used;
no equality expressing preservation of joins is required.
\end{remark}

\begin{remark}[Choice and regularity]
Countable choice for reals supplies the preimage choices in
Lemmas~\ref{lem:range} and \ref{lem:directed} and the choice of a
dense set avoiding eventually constant branches in
Corollary~\ref{cor:perfect-directed}.
The continuous maps and their real codes are obtained by constructions
on a fixed countable basis. The proof makes no appeal to the cone
dichotomy, to a Slaman--Woodin representation theorem, or to a global
real lift of $\varphi$. Its perfect-set coding input is the original
Groszek--Slaman lemma, Theorem~\ref{thm:GS}; the relative highness
argument needed to apply it twice is proved in
Lemma~\ref{lem:high-cover}.
\end{remark}

\section{Determinacy and the Solovay model}
\label{sec:applications}
We obtain the following theorem of Lutz and Siskind.
\begin{theorem}[Lutz and Siskind \cite{LutzSiskind2025}]\label{cor:AD}
In $\mathsf{ZF}+\mathsf{AD}+\DCR$, every order-preserving
map $\varphi:\DT\to\DT$ is constant on a cone or above the
identity on a cone.\footnote{$\mathsf{ZF}+\mathsf{AD}+\DCR$  is the assumption in \cite{LutzSiskind2025}. Actually $\mathsf{ZF}+\mathsf{AD}$ is sufficient. Also what they prove is for a degree invariant function $\varphi$.}
\end{theorem}

\begin{proof}
Dependent choice for reals implies $\ACR$. Determinacy supplies
$\PSP$ and $\CU$. For the latter, comeager uniformization
and its continuous restriction are stated in
\cite[Fact~2.22]{ChanMeehan2018}. Covering the exceptional meager
set by a meager $F_\sigma$ set gives a comeager $G_\delta$ domain.
Apply Theorem~\ref{thm:main}.
\end{proof}

\begin{corollary}[The choiceless Solovay model]\label{cor:Solovay}
Inside the choiceless Solovay model, every order-preserving map of
the Turing degrees belonging to that model is constant on a cone or
above the identity on a cone.
\end{corollary}

\begin{proof}
Solovay's Theorem~1 \cite{Solovay1970} supplies dependent choice,
the perfect-set property, and Borel uniformization with a meager
exceptional set (See clauses (1), (4) and 5(b) in the paper). Transfer this statement to $\Cantor$ using a
homeomorphism between comeager subspaces of $\R$ and $\Cantor$
and a Borel encoding of the output space. Restrict the uniformizer
first to a comeager
$G_\delta$ domain and then to a comeager $G_\delta$ set on which it
is continuous, using Lemma~\ref{lem:continuous-restriction}.
This gives $\CU$, so Theorem~\ref{thm:main} applies internally to
every such map.
\end{proof}

\section{Equiconsistency with an inaccessible cardinal}
\label{sec:consistency}

In this section the order-preserving dichotomy is considered as an
axiom in its own right. No determinacy hypothesis is included in its
formulation. All functions quantified over are functions belonging
to the universe under consideration.

\subsection{The two formulations}

\begin{definition}\label{def:OP}
The principle $\OPreal$ asserts that for every total Turing-invariant,
order-preserving function $f:\Cantor\to\Cantor$, at least one of
the following holds:
\begin{enumerate}[label=(\roman*)]
 \item there are reals $B,Y$ such that
 $f(X)\equivT Y$ for every $X\geT B$;
 \item there is a real $B$ such that
 $X\leT f(X)$ for every $X\geT B$.
\end{enumerate}
The principle $\OPdeg$ is the corresponding assertion for every
order-preserving map $\varphi:\DT\to\DT$.
\end{definition}

Here order preservation for a real function means
$X\leT Y\Longrightarrow f(X)\leT f(Y)$; in particular, it implies
Turing invariance. We retain both adjectives to match the usual
formulation of Martin's conjecture.

In $\mathsf{ZF}$, every function in the real formulation induces the
degree map $\degT{X}\mapsto\degT{f(X)}$. Consequently,
\begin{equation}\label{eq:OP-implication}
 \OPdeg\ \Longrightarrow\ \OPreal.
\end{equation}
The proof below does not require the converse. In particular, it
does not assume that an arbitrary degree map has a global real lift.
The upper bound establishes the degree-map principle, whereas the
lower bound already follows from the real-function principle.

\subsection{The lower bound from an inner-model counterexample}

\begin{theorem}[The obstruction at $\omega_1^{L[r]}=\omega_1$]
\label{thm:inaccessible-reals}
In $\mathsf{ZF}$,
\begin{equation}\label{eq:inaccessible-reals}
 \OPreal\quad\Longrightarrow\quad
 (\forall r\in\Cantor)\ \omega_1^{L[r]}<\omega_1.
\end{equation}
\end{theorem}

\begin{proof}
Suppose that $r\in\Cantor$ satisfies
\[
 \kappa=\omega_1^{L[r]}=\omega_1^V,
\]
and put $M=L[r]$. We construct a counterexample to $\OPreal$
on all reals of $V$.

\smallskip\noindent
\emph{The sequence chosen in $M$.}
Work inside $M$. Since every $0<\alpha<\kappa$ is countable
there, use choice in $M$ to fix surjections
$e_\alpha:\omega\to\alpha$ in $L[r]$. Define by transfinite recursion
\[
 A_0=0^\omega,\qquad
 A_\alpha=\left(\bigoplus_{n\in\omega}
                  A_{e_\alpha(n)}\right)'
       \quad(0<\alpha<\kappa),
\]
where the prime denotes the Turing jump. Every earlier real is
computable from the join at stage $\alpha$, and the jump is
strictly above that join. Thus
\begin{equation}\label{eq:old-chain}
 A_\alpha<_T A_\beta\qquad(\alpha<\beta<\kappa).
\end{equation}
Turing reducibility between these fixed reals is arithmetic and
hence absolute between $M$ and $V$. Therefore
\eqref{eq:old-chain} continues to hold in $V$. The entire sequence
belongs to $M$; no choice indexed by the ambient reals is needed.

\smallskip\noindent
\emph{The function on all ambient reals.}
For $X\in(\Cantor)^V$, define
\begin{equation}\label{eq:least-missing}
 \alpha_X=\min\{\alpha<\kappa:A_\alpha\nleq_T X\},
 \qquad F(X)=A_{\alpha_X}.
\end{equation}
The minimum exists due to uncountability of $\kappa$. Definition~\eqref{eq:least-missing} therefore
gives a total real function in $V$, including at inputs outside $M$.
It uses only the fixed sequence and least ordinals.

\smallskip\noindent
\emph{Order preservation and Turing invariance.}
If $X\leT Y$, every $A_\beta$ with $\beta<\alpha_X$ is computable
from $X$ and hence from $Y$. Thus $\alpha_X\le\alpha_Y$, and the
increasing chain gives $F(X)\leT F(Y)$. If $X\equivT Y$, the
sets of indices computed by the two inputs coincide; in fact
$\alpha_X=\alpha_Y$ and $F(X)=F(Y)$.

\smallskip\noindent
\emph{Failure of both cone conclusions.}
By its definition,
\begin{equation}\label{eq:output-not-below-input}
 F(X)\nleq_T X\qquad(X\in\Cantor).
\end{equation}
If $F$ were constant in degree with value $\degT{Y}$ on the cone
above $\degT{B}$, take $X=B\oplus Y$. Then
$F(X)\equivT Y\leT X$, contradicting
\eqref{eq:output-not-below-input}.

Suppose next that $X\leT F(X)$ for every $X\geT B$, for some
real $B$. Put $\alpha=\alpha_B$. Evaluating at $B$ gives
$B\leT F(B)=A_\alpha$. Since
$A_\alpha<_T A_{\alpha+1}$, choose inside $M$ a real
$X\geT A_\alpha$ incomparable with $A_{\alpha+1}$ (see \cite{Sacks1961}).
These relations remain true in $V$. In particular, $X\geT B$,
every $A_\beta$ with $\beta\le\alpha$ is computable from $X$,
and $A_{\alpha+1}$ is not. Hence
\begin{equation}\label{eq:input-not-below-output}
 \alpha_X=\alpha+1,\qquad F(X)=A_{\alpha+1},\qquad
 X\nleq_T F(X).
\end{equation}
This contradicts the assumed domination on the cone above $B$.
Thus $F$ satisfies neither conclusion of $\OPreal$.

Finally, $\omega_1^{L[r]}\le\omega_1^V$ always holds, since an
ordinal countable in $L[r]$ is countable in $V$. The equality case
having been excluded, \eqref{eq:inaccessible-reals} follows.
\end{proof}

\begin{lemma}[See Proposition~11.5, p.~135, \cite{Kanamori2003}]\label{lem:inaccessible-L}
Work in $\mathsf{ZF}+\ACR$\footnote{We need $\ACR$ to prove regularity of $\omega_1$ here.}. If
$\omega_1^{L[r]}<\omega_1$ for every real $r$, then
$\omega_1^V$ is a strongly inaccessible cardinal in $L$.
\end{lemma}

The condition in \eqref{eq:inaccessible-reals} is often expressed
by saying that $\omega_1$ is \emph{inaccessible to reals}. The
last lemma includes the regularity hypothesis through $\ACR$;
this point matters in a choiceless base theory.

\begin{corollary}[The lower consistency bound]\label{cor:lower-bound}
Over $\mathsf{ZF}+\ACR$, either $\OPreal$ or $\OPdeg$ implies
\[
 L\models\text{``$\omega_1^V$ is strongly inaccessible.''}
\]
In particular, either principle gives an inner model of
$\mathsf{ZFC}$ with a strongly inaccessible cardinal.
\end{corollary}

\begin{proof}
Combine \eqref{eq:OP-implication},
Theorem~\ref{thm:inaccessible-reals}, and
Lemma~\ref{lem:inaccessible-L}.
\end{proof}

\subsection{The upper bound and the equiconsistency}

\begin{proof}[Proof of Theorem~\ref{thm:consistency}]
Write $I$ for the assertion that a strongly inaccessible cardinal
exists. Full $\mathsf{DC}$ implies every permitted base principle
$B$. The Solovay construction and Corollary~\ref{cor:Solovay}
therefore give
\[
 \Con(\mathsf{ZFC}+I)
 \ \Longrightarrow\ \Con(\mathsf{ZF}+B+\OPdeg).
\]
Equation~\eqref{eq:OP-implication} then gives the implication from
the consistency of~\textup{(iii)} to that of~\textup{(ii)}.
Conversely, each permitted $B$ implies $\ACR$, so
Corollary~\ref{cor:lower-bound} gives
\[
 \Con(\mathsf{ZF}+B+\OPreal)
 \ \Longrightarrow\ \Con(\mathsf{ZFC}+I).
\]
These implications establish the claimed equiconsistencies. The
additional assertion about $\omega_1^V$ follows from
Corollary~\ref{cor:lower-bound}.
\end{proof}

\section*{Acknowledgments}
The author formulated the research program and sketched the proof
strategy. ChatGPT (OpenAI) assisted in developing and filling in
technical details of the proofs, checking intermediate arguments
and references, and preparing the exposition and \LaTeX{} manuscript.
The author is responsible for the final mathematical content.
 
\Needspace{17\baselineskip}


\begin{thebibliography}{9}

\bibitem{ChanMeehan2018}
W.~Chan and C.~Meehan,
\emph{Definable combinatorics of some Borel equivalence relations},
preprint, arXiv:1709.04567; author version dated March 11, 2018.
\url{https://williamchan-math.github.io/papers/definable_combinatorics_equivalence_relation.pdf}.

\bibitem{GroszekSlaman1998}
M.~J.~Groszek and T.~A.~Slaman,
\emph{A basis theorem for perfect sets},
The Bulletin of Symbolic Logic \textbf{4} (1998), no.~2, 204--209.
\href{https://doi.org/10.2307/421023}{doi:10.2307/421023}.
Author version: \url{https://math.berkeley.edu/~slaman/papers/prikry.pdf}.

\bibitem{Jockusch1972}
C.~G.~Jockusch, Jr.,
\emph{Degrees in which the recursive sets are uniformly recursive},
Canadian Journal of Mathematics \textbf{24} (1972),
no.~6, 1092--1099.
\href{https://doi.org/10.4153/CJM-1972-113-9}
{doi:10.4153/CJM-1972-113-9}.

\bibitem{Soare2016}
R.~I.~Soare,
\emph{Turing Computability: Theory and Applications},
Theory and Applications of Computability,
Springer, Berlin--Heidelberg, 2016.
\href{https://doi.org/10.1007/978-3-642-31933-4}
{doi:10.1007/978-3-642-31933-4}.

\bibitem{Kanamori2003}
A.~Kanamori,
\emph{The Higher Infinite:
Large Cardinals in Set Theory from Their Beginnings},
second edition,
Springer Monographs in Mathematics,
Springer-Verlag, Berlin, 2003.

\bibitem{Kechris1995}
A.~S.~Kechris,
\emph{Classical Descriptive Set Theory},
Graduate Texts in Mathematics, vol.~156,
Springer-Verlag, New York, 1995.
\href{https://doi.org/10.1007/978-1-4612-4190-4}
{doi:10.1007/978-1-4612-4190-4}

\bibitem{LutzSiskind2025}
P.~Lutz and B.~Siskind,
\emph{Part 1 of Martin's Conjecture for order-preserving and
measure-preserving functions},
Journal of the American Mathematical Society \textbf{38} (2025),
no.~2, 319--367.
\href{https://doi.org/10.1090/jams/1046}{doi:10.1090/jams/1046}.
Preprint: \url{https://arxiv.org/abs/2305.19646v3}.
Theorem numbering follows this preprint version.

\bibitem{Martin1966}
D.~A.~Martin,
\emph{Classes of recursively enumerable sets and degrees of unsolvability},
Zeitschrift f\"ur mathematische Logik und Grundlagen der Mathematik
\textbf{12} (1966), 295--310.
\href{https://doi.org/10.1002/malq.19660120125}{doi:10.1002/malq.19660120125}.

\bibitem{Sacks1961}
G.~E.~Sacks,
\emph{On suborderings of degrees of recursive unsolvability},
Zeitschrift f\"ur mathematische Logik und Grundlagen der Mathematik
\textbf{7} (1961), 46--56.

\bibitem{Solovay1970}
R.~M.~Solovay,
\emph{A model of set-theory in which every set of reals is Lebesgue
measurable},
Annals of Mathematics (2) \textbf{92} (1970), no.~1, 1--56.
\href{https://doi.org/10.2307/1970696}{doi:10.2307/1970696}.

\end{thebibliography}
\end{document}